\documentclass[11pt]{amsart}
\usepackage{latexsym}
\usepackage{amsfonts}
\usepackage{xcolor}
\usepackage{graphicx}

\newcommand{\field}[1]{\mathbb{#1}}
\newcommand{\C}{\field{C}}

\newtheorem{defi}{Definition}[section]
\newtheorem{lem}[defi]{Lemma}

\newtheorem{theo}[defi]{Theorem}
\newtheorem{co}[defi]{Corollary}
\newtheorem{pr}[defi]{Proposition}
\newtheorem{re}[defi]{Remark}
\newtheorem{ex}[defi]{Example}

\font\tenmsy=msbm10

\def\Bbb#1{\hbox{\tenmsy#1}} 
\title[On algebraic bi-Lipschitz homeomorphisms ]{On algebraic  bi-Lipschitz homeomorphisms  }
\makeatletter

\@addtoreset{equation}{section}
\makeatother

\author{Zbigniew Jelonek}
\address[Zbigniew  Jelonek]{Instytut Matematyczny\\
Polska Akademia Nauk\\
\'Sniadeckich 8, 00-656 Warszawa\\
Poland }
\email{najelone@cyf-kr.edu.pl}


\subjclass{14 R 10, 14 P 15, 14 P 10}
\thanks{The author is partially supported by the grant of Narodowe Centrum Nauki number 2019/33/B/ST1/00755}

\date{}

\begin{document}


\begin{abstract}

Let $X\subset \C^n; Y\subset \C^m$ be closed affine varieties and 
let $\phi: X\to Y$ be an algebraic bi-Lipschitz homeomorphism. Then ${\rm deg}\ X={\rm deg}\ Y.$
Similarly, let $(X,0)\subset (\C^n,0), (Y,0)\subset (\C^m,0)$ be germs of analytic  sets and let $f: (X,0)\to (Y,0)$ be a c-holomorphic  and bi-Lipschitz homeomorphism.  Then
${\rm mult}_0 \ X= {\rm mult }_0 \ Y.$ Finally we show that the normality is not a bi-Lipschitz invariant.
\end{abstract}

\bibliographystyle{alpha}

\maketitle

\section{Introduction}
In \cite{bfs} Bobadilla, Fernandes and Sampaio study invariance of degree of complex affine varieties under bi-Lipschitz homeomorphisms. They proved that this problem is equivalent to the bi-Lipschitz version of the Zariski multiplicity conjecture. Moreover, they proved that the degree of curves and surfaces are such invariants. From other point of view in \cite{bfsv} Birbrair, Fernandes, Sampaio and Verbitsky give an example of two three-dimensional affine varieties, which are bi-Lipschitz equivalent but they have different degrees. In fact they showed that we have two different embeddings of   $\C\Bbb P^1\times \C\Bbb P^1$ into $\C\Bbb P^5$, say $X$ and $Y$, such that affine cones $C(X),C(Y)\subset \C^6$ are bi-Lipschitz equivalent, but they have different degrees.  Hence in general the algebraic degree of an affine set is not a bi-Lipschitz invariant. However varieties $X,Y$ of  Birbrair, Fernandes, Sampaio and Verbitsky have codimension greater than one. Hence the problem of invariance of degree under bi-Lipschitz homeomorphisms is still open in the important case of  affine hypersurfaces in $\C^n$, where $n>3.$ 

In this paper we introduce a new class of bi-Lipschitz homeomorphism- the algebraic bi-Lipschitz homeomorphisms. Assume that $X,Y$ are affine varieties.  We say that homeomorphism  $f:X\to Y$ is an algebraic bi-Lipschitz homeomorphism if it is a bi-Lipschitz  and its graph is an algebraic variety. Inspired by our recent paper \cite{bfj} (with L. Birbrair and A. Fernandes) we show that this class of mappings preserves the degree of  affine varieties:

\vspace{5mm}

\noindent {\bf Theorem 3.4}
{\it Let $X\subset \C^n,Y\subset \C^m$ be affine algebraic varieties and let $f: X\to Y$ be an algebraic bi-Lipschitz homeomorphism. Then
${\rm deg} \ X= {\rm deg }\ Y.$}

In a similar way we can prove:

\noindent {\bf Theorem 4.1}
{\it Let $(X,0)\subset (\C^n,0), (Y,0)\subset (\C^m,0)$ be germs of analytic  sets and let $f: (X,0)\to (Y,0)$ be a c-holomorphic  and bi-Lipschitz homeomorphism.  Then
${\rm mult}_0 \ X= {\rm mult }_0 \ Y.$}

Let us recall that a mapping $f:X\to Y$ is c-holomorphic, if it is continuous and its graph is analytic in $X\times Y$ (here $X,Y$ are analytic sets).

\noindent {\bf Corollary 4.2}
{\it Let $X\subset \C^n,Y\subset \C^m$ be affine algebraic varieties and let $f: X\to Y$ be an algebraic bi-Lipschitz homeomorphism. Then
${\rm deg} \ X= {\rm deg }\ Y$ and for every $x\in X$ we have ${\rm mult}_x \ X= {\rm mult }_{f(x)} \ Y.$}

At the end of this paper we give an example of two algebraic affine varieties $X,Y$ such that $X$ is normal, $Y$ is not normal and there exists an algebraic bi-Lipschitz homeomorphism $f: X\to Y.$ Hence the normality is not a bi-Lipschitz invariant.

\section{Preliminaries}

\begin{defi}
Let $X,Y$ be affine complex varieties and let $f: X\to Y$ be a continuous mapping. We say that $f$ {\rm is algebraic} if the graph of 
$f$ is a complex algebraic set. 
\end{defi}

\begin{pr}\label{alg}
Every algebraic mapping is rational, i.e., there exists a Zariski open subset $U\subset X$ such that
the mapping $f|_{U}$ is a regular mapping. 
\end{pr}

\begin{proof}
Indeed, let $\Gamma$ be a graph of $f$ and $\iota : X\ni x \mapsto (x,f(x))\in \Gamma.$
Let $U\subset X$ be the set of smooth points of $X.$ Then $U$ is a Zariski open dense subset of $X.$ By the Zariski Main Theorem the projection
$\pi : \Gamma \to X$ is an isomorphism over $U.$ This means that $\iota$ is an isomorphism on $U.$ But $f|U=\pi\circ\iota|_U,$ hence $f$ is regular on $U.$ 
\end{proof}

In other words an algebraic mapping is a mapping which additionally is rational. Here we are interested in these algebraic mappings, which are additionally bi-Lipschitz. Note that if a mapping $f: X\to Y$ is an algebraic bi-Lipschitz homeomorphism, then the mapping $f^{-1}: Y\to X$ is also an algebraic bi-Lipschitz homeomorphism. First of all let us recall the following statement, which follows from \cite{bfj}:

\begin{theo}\label{proj}
Let $X\subset \C^n$ be a $k-$dimensional affine variety. Then there is an algebraic bi-Lipschitz embedding $f: X\to\C^{2k+1}.$
\end{theo}

We say that the mapping $f: X \to \C^n$ is an algebraic bi-Lipschitz embedding, if the mapping $f: X\to f(X)$ is algebraic and bi-Lipschitz. 
Since not every $k$ dimensional affine variety can be embedded into $\C^{2k+1}$ in a bi-regular way (see example below) we see that we have a lot of algebraic bi-Lipschitz homeomorphisms which are not biregular mappings.

\begin{ex}
{\rm For $m>3$ let $X\subset \C^{m}$ be a curve given by parametrization $$X=\{ x\in \C^m: x=(t^m,...., t^{2m-1}); \ t\in \C\}.$$ It is an easy observation that the Zariski tangent space $T_0 X$ coincide with $\C^m.$ Now consider a generic linear projection $\pi: X\to \C^3.$ By \cite{bfj}  it is an algebraic bi-Lipschitz embedding.
Since $T_0(\pi(X))\subset \C^3$ and $T_0 X=\C^m$ we see that the mapping $\pi$ is not bi-regular. In fact the algebraic bi-Lipschitz mapping $\pi^{-1}: \pi(X)\to X$ can not be extended to any $C^1$ mapping in a neighborhood of $0$ in $\C^3$.}
\end{ex}

\begin{ex}
{\rm Let $X=\{ x\in \C^3 : x=(t, t^3+t^2,t^5) ; t\in \C\}$ and $Y=\{x\in\C^3 : x=(t, t^3+2t^2,t^5); t\in \C\}.$ Let $\phi : X\ni (t, t^3+t^2,t^5)\mapsto (t, t^3+2t^2,t^5)\in Y.$ Then $\phi$ is bi-regular and bi-Lipschitz but it is not a linear mapping.

It is bi-regular because $\phi(x,y,z)=(x,y+x^2,z).$ Now 
we show that $\phi$ is bi-Lipschitz outside some big ball. Let $a(t)=(t, t^3+t^2,t^5), b(t)=(t, t^3+2t^2,t^5).$ Hence $\phi(a(t))=b(t)$. We have to show that for some $K>0$ $$\frac{1}{K}||a(t)-a(s)||<||b(t)-b(s)||< K||a(t)-a(s)||.$$ Consider the fraction $$(*)\ \frac{b_2(t)-b_2(s))}{a_2(t)-a_2(s)}=\frac{t^3+2t^2-s^3-2s^2}{t^3+t^2-s^3-s^2}=\frac{(1-\epsilon_1\alpha)(1-\epsilon_2\alpha)+2/t(1+\alpha)}{(1-\epsilon_1\alpha)(1-\epsilon_2\alpha)+1/t(1+\alpha)},$$ where $1,\epsilon_1,\epsilon_2$ are all roots of polynomial $x^3+1$ and $\alpha=s/t.$ Note that we can always assume that $|\alpha|\le 1.$ Denote  $g(\alpha)=(1-\epsilon_1\alpha)(1-\epsilon_2\alpha).$ This polynomial has roots in $\epsilon_1,\epsilon_2$ only. Denote by  $U=\{ \alpha: |\alpha-\epsilon_1|>r\} \cap \{\alpha: |\alpha-\epsilon_2|>r\}.$ Since $U$ does not contain roots of $g$, we have  $|g(\alpha)|>\rho>0.$
We can assume that $r$ is so small that in the set $V:=\{ \alpha: |\alpha-\epsilon_1|\le r\}\cup \{\alpha: |\alpha-\epsilon_2|\le r\}$ there is no roots of polynomials $x^5+1.$ Now we estimate $(*)$. We have two possibilities: 

a) $\alpha\in U$,

b) $\alpha\in V.$

\noindent In the case a) we have $|g(\alpha)| > \rho>0$ and in particular for $|t|>R$ we have 
$$\frac{||b_2(t)-b_2(s))||}{||a_2(t)-a_2(s)||}<2.$$
Hence $$||b_2(t)-b_2(s))||<2||a_2(t)-a_2(s)||.$$

In the case b) we have $||\phi(a(t))-\phi(b(t))||=||b_3(t)-b_3(s)||=||a_3(t)-a_3(s))||=||a(t)-a(s)||$ for large $|t|$ (we consider here the "sup" norm). Hence indeed the mapping $\phi$ is Lipschitz outside a large ball. Similarly $\phi^{-1}$ is Lipschitz outside a large ball.
Hence $\phi$ is bi-Lipschitz outside a large ball.

On the other hand $\phi$ is bi-Lipschitz in any ball, because it is a smooth mapping. Combining this fact with the first step of our proof we see that we can reduce the general case to the case where $|\alpha|$ is small and $|t|$ is large. But this can be done in a similar way as in the first step (we left details to the reader). Hence finally the mapping $\phi$ is bi-Lipschitz.

The mapping $\phi$ is not a restriction of a linear mapping because otherwise $t^3+2t^2=at+b(t^3+t^2)+ct^5+d$, where $a,b,c,d\in \C.$ This is impossible.}
\end{ex}

 \section{Proof of the  Theorem 3.4}

\begin{defi}
Let $L^s, H^{n-s-1}$ be two disjoint linear subspaces of $\Bbb CP^n.$ Let $\pi_\infty$ be a hyperplane (a hyperplane at infinty) and assume that
$L^s \subset \pi_\infty.$ By a projection $\pi_L$ with center $L^s$ we mean the mapping:
$$ \pi_L : \C^n=\C\Bbb P^n\setminus \pi_\infty\ni x \mapsto <L^s,x> \cap H^{n-s-1}\in H^{n-s-1}\setminus \pi_\infty=\C^{n-s-1}.$$ Here by $<L,x>$
we mean a linear subspace spanned by $L$ and $\{x\}.$
\end{defi}

\begin{lem}\label{lemat}
Let $X$ be a closed subset of $\C^n.$ Denote by $\Lambda\subset \pi_\infty$ the set of directions of all secants of $X$ and let $\Sigma= \overline{\Lambda},$
where $\pi_\infty$ is a hyperplane at infinity and we consider the projective closure. Let $\pi_L:\C^n\to \C^l$ be a projection with center $L$.
Then $\pi_L|_X$ is a bi-Lipschitz embedding if and only if $L\cap \Sigma=\emptyset.$
\end{lem}

\begin{proof}
a) Assume that $L\cap \Sigma=\emptyset.$ We will proceed by induction.
Since a linear affine mapping is a bi-Lipschitz homeomorphism, we can assume that $\pi_L$ coincide with the projection
$\pi: \C^n\ni (x_1,...,x_n) \mapsto (x_1,...,x_k,0,...,0)\in \C^k\times \{0,...,0\}.$ We can decompose $\pi$ into two projections:
$\pi=\pi_{2}\circ \pi_{1}$, where $\pi_1:\C^n \ni (x_1,...,x_n)\mapsto (x_1,...,x_{n-1},0)=\C^{n-1}\times 0$ is a projection with a center $P_1=(0:0:...:1)$
and $\pi_2 :\C^{n-1}\ni (x_1,...,x_{n-1},0)\mapsto (x_1,...,x_{k},0,...,0)\in \C^{k}\times \{(0,...,0)\}$ is a projection with a center $L':=\{ x_0=0,...,x_{k}=0\}.$ Since $P_1\in L$ and consequently $P_1\not\in \Sigma$, we prove  that $\pi_1$ is an algebraic bi-Lipschitz homeomorphism.
Indeed, let $P_1\in \Bbb {CP}^{n-1}\setminus \Sigma$ and let $H\subset
\C^{n}$ be a hyperplane, such that $P\not\in H.$ Since a complex linear isomorphism is a bi-Lipschitz homeomorphism, we can assume that
$P_1=(0:0:...0:1)$ and $H=\{x_n=0\}.$ We show that the projection $p: X\to H$ with center at $P_1$ is an algebraic bi-Lipschitz homeomorphism.
Of course $||p(x)-p(y)||\le ||x-y||.$ Assume that $p$ is not bi-Lipschitz, i,e., there is a sequence of points $x_j,y_j\in X$
such that $$\frac{||p(x_j)-p(y_j)||}{||x_j-y_j||}\to 0,$$ as $n\to \infty.$ Let $x_j-y_j=(a_1(j),...,a_{n-1}(j),b(j))$ and denote by
$P_j$ the corresponding point $(a_1(j):...:a_{n-1}(j):b(j))$ in $\Bbb
{CP}^{n-1}.$ Hence $$P_j=\frac{(a_1(j):...:a_{n-1}(j):b(j))}{||x_j-y_j||}.$$ Since $\frac{(a_1(j),...,a_{n-1}(j))}{||x_j-y_j||}=
\frac{p(x_j)-p(y_j)}{||x_j-y_j||}\to 0$, we get that $P_j\to P.$ It is a contradiction.
Notice that if $\pi_1(X)=X'$, then $\Sigma'=\pi_1(\Sigma).$ Moreover $L'=L\cap \{ x_n=0\}$ and $<L', P_1>=L$. This means that $\Sigma'\cap L'=\emptyset.$ Now we can finish by induction.

\vspace{5mm}

b) Assume that $\pi_L|_X$ is a bi-Lipschitz map and $\Sigma \cap L\not=\emptyset.$ As before we can change a system of coordinates in such a way that
$\pi: \C^n\ni (x_1,...,x_n) \mapsto (x_1,...,x_k,0,...,0)\in \C^k\times \{0,...,0\}.$ Moreover, we can assume that $P_1=(0:0:...:1)\in \Sigma.$
Actually $\pi_1$ is not bi-Lipschitz. Indeed there is a sequence of secants $l_n=(x_n,y_n)$ of $X$ whose directions tends to $P_1.$
 Let $x_j-y_j=(a_1(j),...,a_{n-1}(j),b(j))$ and denote by
$P_j$ the corresponding point $(a_1(j):...:a_{n-1}(j):b(j))$ in $\Bbb
{CP}^{n-1}.$ Hence $$P_j=\frac{(a_1(j):...:a_{n-1}(j):b(j))}{||x_j-y_j||}.$$ Since  $P_j\to P$ we have $\frac{(a_1(j),...,a_{n-1}(j))}{||x_j-y_j||}=
\frac{p(x_j)-p(y_j)}{||x_j-y_j||}\to 0$. Hence the mapping $\pi_1$ is not bi-Lipschitz. 

 Now it is enough to note that $||\pi_2(x)-\pi_2(y)||\le ||x-y||,$ hence $||\pi(x_n)-\pi_(y_n)||=||\pi_2(\pi_1(x_n))-\pi_2(\pi_1(y_n))||\le ||\pi_1(x_n)-\pi_1(y_n)||$. Thus  $$\frac{||x_n-y_n||}{||\pi(x_n)-\pi(y_n)||}\ge \frac{||x_n-y_n||}{||\pi_1(x_n)-\pi_1(y_n)||}\to \infty.$$ This contradiction finishes the proof.
\end{proof}

\begin{lem}\label{wykres}
Let $X\subset \C^{n}$ be  a closed  set and let $f: X\to\C^m$ be an algebraic  Lipschitz homeomorphism. Let $Y:=graph(f)\subset \C^n\times \C^m.$
Then the mapping $\phi: X\ni x \mapsto (x,f(x)) \in Y$ is an algebraic  bi-Lipschitz homeomorphism.
\end{lem}

\begin{proof}
The mapping $\phi$ is algebraic by Proposition \ref{alg}.
Since $f$ is Lipschitz, there is a constant $C$ such that $$||f(x)-f(y)||<C||x-y||.$$
We have $$||\phi(x)-\phi(y)||=||(x-y,f(x)-f(y))||\le$$  $$ \le ||x-y||+||f(x)-f(y)||\le ||x-y||+C||x-y||\le (1+C)||x-y||.$$
Moreover $$||x-y|| \le ||\phi(x)-\phi(y)||.$$ Hence $$||x-y||\le ||\phi(x)-\phi(y)||\le (1+C)||x-y||.$$
\end{proof}

\begin{theo}
Let $X\subset \C^n,Y\subset \C^m$ be affine algebraic varieties and let $f: X\to Y$ be an algebraic bi-Lipschitz homeomorphism. Then
${\rm deg} \ X= {\rm deg }\ Y.$
\end{theo}

\begin{proof}  Let $X\subset \C^n$ and Let $f: X\to Y$ be an algebraic bi-Lipschitz homeomorphism. Denote by $\Gamma\subset \C^n\times\C^n$ the graph of $f.$ By Lemma \ref{wykres} projections $\pi_X:\Gamma\to X$ and $\pi_Y:\Gamma\to Y$ are algebraic bi-Lipschitz homeomorphism. Let
$\pi_1: \C^n\times \C^m\to \C^n$ and $\pi_2: \C^n\times \C^m\to \C^m$ and denote by $S_1,S_2\subset \pi_\infty=\C\Bbb P^{m+n-1}$ centers of these projections. Denote by $\Lambda\subset \pi_\infty$ the set of directions of all secants of $\Gamma$ and let $\Sigma= cl(\Lambda)$.
Since $\pi_1|_\Gamma=\pi_X$ and $\pi_2|_\Gamma=\pi_Y$ we have by Lemma \ref{lemat} that $\Sigma \cap S_1=\Sigma\cap S_2=\emptyset.$ 
But $\overline{\Gamma}\setminus \Gamma\subset \Sigma$ and consequently $\overline{\Gamma}\cap S_i=\emptyset$ for $i=1,2.$ Now let $L\subset \C^n$ be a generic linear subspace of dimension $k={\rm codim } \ X.$ Hence $\#(L\cap X)={\rm deg} \ X$ and $L$ has not common points with $X$ at infinity.
Since $\overline{\Gamma}\cap S_1=\emptyset$ we see that $\#(<S_1,L>\cap\ \Gamma)={\rm deg}\ \Gamma$ where  by $<S_1,L>$
we mean a linear (projective) subspace spanned by $L$ and $S_1.$ However the mapping $\pi_X$ is a bijection, hence $\#(<S_1,L>\cap\ \Gamma)={\rm deg}\ X.$ In particular ${\rm deg}\ \Gamma={\rm deg}\ X.$ In the same way ${\rm deg}\ \Gamma={\rm deg}\ Y.$ Hence  ${\rm deg}\ X={\rm deg} \ Y.$
\end{proof}

\begin{re}
{\rm In fact a more general statement  is true. We can say after \cite{bfs} that the mapping $f: X\to Y$ is bi-Lipschitz at infinity if there exist compact sets $K,K'$ such that the mapping $f': X\setminus K\ni x\mapsto f(x)\in Y\setminus K'$ is bi-Lipschitz. It is easy to see that our proof works if $f$ is an algebraic homeomorphism, which is bi-Lipschitz at infinity. Indeed under this assumption we still have $\overline{\Gamma}\setminus \Gamma \subset \Sigma$, where $\Gamma=graph(f)$ and $\Sigma$ is the set of directions of secants of $graph(f').$ Moreover, there are sufficiently general linear subspaces which omit $K$ or $K'.$ Hence in fact we can state:

\begin{theo}
Let $X\subset \C^n,Y\subset \C^m$ be affine algebraic varieties and let $f: X\to Y$ be a birational correspondence. Assume that there exist compact sets $K,K'$ such that $f: X\setminus K\to Y\setminus K'$ is defined (as continuous mapping) and bi-Lipschitz. Then
${\rm deg} \ X= {\rm deg }\ Y.$
\end{theo}

}
\end{re}

\begin{re}
{\rm 
It is worth to say, that affine cones $C(X),C(Y)$ mentioned in the introduction are birationally and bi-Lipschitz equivalent, but they have different degrees.}
\end{re}

 \section{Proof of the  Theorem 4.1} 

In a similar way we can prove:

\begin{theo}
Let $(X,0)\subset (\C^n,0), (Y,0)\subset (\C^m,0)$ be germs of analytic  sets and let $f: (X,0)\to (Y,0)$ be a c-holomorphic  and bi-Lipschitz homeomorphism.  Then
${\rm mult}_0 \ X= {\rm mult }_0 \ Y.$
\end{theo}.

\begin{proof}
Let $U,V$ be small neighborhoods of $0$ in $\C^n$ and $\C^m$ such that the mapping $f:U\cap X=X'\to V\cap Y=Y'$ is defined and it is by-Lipschitz.
Denote by $\Gamma\subset U\times V$ the graph of $f.$ By Lemma \ref{wykres} projections $\pi_{X'}:\Gamma\to X'$ and $\pi_{Y'}:\Gamma\to Y'$ are   bi-Lipschitz homeomorphism. Let
$\pi_1: \C^n\times \C^m\to \C^n$ and $\pi_2: \C^n\times \C^m\to \C^m$ and denote by $S_1,S_2\subset \pi_\infty=\C\Bbb P^{m+n-1}$ centers of these projections. Denote by $\Lambda\subset \pi_\infty$ the set of directions of all secants of $\Gamma$ and let $\Sigma= cl(\Lambda)$.
Since $\pi_1|_\Gamma=\pi_X$ and $\pi_2|_\Gamma=\pi_Y$ we have by Lemma \ref{lemat} that $\Sigma \cap S_1=\Sigma\cap S_2=\emptyset.$ 
But $\overline{C(0,\Gamma)}\setminus C(0,\Gamma)\subset \Sigma$ and consequently $\overline{C(0,\Gamma)}\cap S_i=\emptyset$ for $i=1,2.$ Now let $L\subset \C^n$ be a generic linear subspace of dimension $k={\rm codim } \ X.$ Hence $\#(L\cap X')={\rm mult}_0 X$ and $L$ has not common points with $C(0,X)$ at infinity (we can shrink $U,V$ if necessary!).
Since $\overline{C(0,\Gamma)}\cap S_1=\emptyset$ we see that $\#(<S_1,L>\cap\ \Gamma)={\rm mult}_0\ \Gamma$ where  by $<S_1,L>$
we mean a linear (projective) subspace spanned by $L$ and $S_1.$ However the mapping $\pi_{X'}$ is a bijection, hence $\#(<S_1,L>\cap\ \Gamma)={\rm mult}_0\ X.$ In particular ${\rm mult}_0\ \Gamma={\rm mult}_0\ X.$ In the same way ${\rm mult}_0\ \Gamma={\rm mult}_0\ Y.$ Hence  ${\rm mult}_0\ X={\rm mult}_0 \ Y.$
\end{proof}

\begin{co}
Let $X\subset \C^n,Y\subset \C^m$ be affine algebraic varieties and let $f: X\to Y$ be an algebraic bi-Lipschitz homeomorphism. Then
${\rm deg} \ X= {\rm deg }\ Y$ and for every $x\in X$ we have ${\rm mult}_x \ X= {\rm mult }_{f(x)} \ Y.$
\end{co}

\section{Normality is not a bi-Lipschitz invariant}

In \cite{bfls} (see also \cite{s}) the authors proved that the smoothness of an analytic space is a bi-Lipschitz invariant. The natural question 
is whether the normality is also a bi-Lipschitz invariant. Directly from \cite{bfls} and \cite{s} we have:

\begin{theo}\label{normal}
Let $(X,0), (Y,0)$ be analytic hypersurfaces in $(\C^n,0).$ If they are bi-Lipschitz equivalent and $X$ is normal , then also $Y$ is normal.
\end{theo}

\begin{proof}
Indeed, such a hypersurface is normal if it is smooth in codimension one. Hence our result follows directly from \cite{bfls} and \cite{s}.
\end{proof}

In this section we show that in general Theorem \ref{normal} does not hold. We start with:

\begin{lem}\label{l1}
Let $X\Bbb P^{n}$ be a projective variety, which is not contained in any hyperplane. Let $C(X)\subset\C^{n+1}$ be a cone over $X.$ Then
${\rm dim}\ T_0(C(X))=n+1.$
\end{lem}

\begin{proof}
Note that $T_0(C(X))$ is the dimension of a minimal smooth germ $Y_0$ which contains the germ  $C(X)_0.$ However since $C(X)_0\subset Y_0$ we have that $C(X)\subset T_0 Y.$ Denote by $\pi\cong \Bbb P^{n}$ the hyperplane at infinity. By the assumption we have $\overline{T_0 Y}\cap \pi=\pi.$ Hence $T_0 Y=\C^{n+1}.$
\end{proof}

\begin{theo}
For every $r>1$ there is a normal affine algebraic variety $X^r$ of dimension $r$ and an algebraic bi-Lipschitz homeomorphism 
$\phi: X^r\to Y^r$ such that $Y^r$ is not normal. 
\end{theo}

\begin{proof}
Let us take $d$-tuple embedding $\phi_d$ of $\Bbb P^{r}$ to $\Bbb P^N(d).$ Denote $A_d=\phi_d(\Bbb P^r).$ Let us note 
that  the space $A_d\subset \Bbb P^N(d)$ is projectively normal i.e., the mapping $$\Gamma(\Bbb P^N(d),{\mathcal O}_{\Bbb P^N(d)} (k))\to \Gamma(A_d,{\mathcal O}_X (k))$$ is surjective for every $k$. This means that the cone $X:=C(A_d)$ is a normal space  (for details see \cite{har}, p.126, 5.14). Moreover it is easy to check that $A_d$ is not contained in any linear subspace of $\Bbb P^N(d).$ Hence by Lemma \ref{l1} we have dim $T_0 X=N(d)+1.$  We can take $d$ so large  that $N:=N(d)>2r+1.$ Now consider a generic projection $\pi : \C^N\to\C^{2r+1}.$ We know by \cite{bfj} that $\pi$ restricted to $X_d$ is by-Lipschitz embedding. Let $Y=\pi(X).$ The variety $Y$ is not normal. Indeed, otherwise by the Zariski Main Theorem the mapping $\pi|X: X\to Y$ has to be an isomorphism , in particular dim $T_0 X=$dim $T_0 Y$. Since dim $T_0 Y\le 2r+1<N$ it is a contradiction.
\end{proof}

    \end{document}